\documentclass[oneside, a4paper, 10pt]{article}
\usepackage[colorlinks = true,
            linkcolor = black,
            urlcolor  = blue,
            citecolor = black]{hyperref}
\usepackage{amsmath, amsthm, amssymb, amsfonts, amscd}
\usepackage[left=2cm,right=2cm,top=3cm,bottom=3cm,a4paper]{geometry}
\usepackage{mathtools}
\usepackage{thmtools}
\usepackage{graphicx}
\usepackage{cleveref}
\usepackage{mathrsfs}
\usepackage{titling}
\usepackage{url}
\usepackage[nosort, noadjust]{cite}
\usepackage{enumitem}
\usepackage{tikz-cd}
\usepackage{multirow}
\usepackage{fancyhdr}
\usepackage{sectsty}
\usepackage{longtable}
\usepackage{authblk}
\DeclareGraphicsExtensions{.pdf,.png,.jpg}

\newcommand{\cok}{{\operatorname{cok}}}

\newcommand{\Cl}{{\operatorname{Cl}}}

\newcommand{\Aut}{\operatorname{Aut}}
\newcommand{\GL}{\operatorname{GL}}

\newcommand{\Sp}{\operatorname{Sp}}
\newcommand{\GSp}{\operatorname{GSp}}

\newcommand{\rank}{\operatorname{rank}}

\newcommand{\Prob}{\operatorname{Prob}}

\newcommand{\M}{\operatorname{M}}

\newcommand{\Z}{\mathbb{Z}}
\newcommand{\Q}{\mathbb{Q}}

\newcommand{\F}{\mathbb{F}}

\theoremstyle{definition}
\newtheorem{theorem}{Theorem}[section]
\newtheorem{proposition}[theorem]{Proposition}

\newtheorem{lemma}[theorem]{Lemma}
\newtheorem{conjecture}[theorem]{Conjecture}
\newtheorem*{conjecture*}{Conjecture}
\newtheorem{remark}[theorem]{Remark}

\newtheorem{example}[theorem]{Example}

\numberwithin{equation}{section}

\makeatletter 
\newcommand\semilarge{\@setfontsize\semilarge{11}{13.2}}
\makeatother

\title{\semilarge{\textbf{JOINT DISTRIBUTION OF THE COKERNELS OF RANDOM $p$-ADIC MATRICES}}}
\author{\normalsize{JUNGIN LEE} }
\date{}

\newcommand\shorttitle{JOINT DISTRIBUTION OF THE COKERNELS OF RANDOM $p$-ADIC MATRICES}
\newcommand\authors{JUNGIN LEE}

\sectionfont{\large}
\subsectionfont{\normalsize}

\renewenvironment{abstract}
 {\quotation\small\noindent\rule{\linewidth}{.5pt}\par\smallskip
  {\centering\bfseries\abstractname\par}\medskip}
 {\par\noindent\rule{\linewidth}{.5pt}\endquotation}
 
\AtEndDocument{%
  \bigskip
  \footnotesize

  \textsc{Jungin Lee, Center for Mathematical Challenges, Korea Institute for Advanced Study, Seoul 02455, Korea}\par\nopagebreak
  \textit{E-mail address:} \texttt{jilee.math@gmail.com} }

\begin{document}
\maketitle
\vspace{-15mm}

\begin{abstract}
In this paper, we study the joint distribution of the cokernels of random $p$-adic matrices. 
Let $p$ be a prime and $P_1(t), \cdots, P_l(t) \in \Z_p[t]$ be monic polynomials whose reductions modulo $p$ in $\F_p[t]$ are distinct and irreducible. 
We determine the limit of the joint distribution of the cokernels $\cok (P_1(A)), \cdots, \cok(P_l(A))$ for a random $n \times n$ matrix $A$ over $\Z_p$ with respect to Haar measure as $n \rightarrow \infty$. 
By applying the linearization of a random matrix model, we also provide a conjecture which generalizes this result.
Finally, we provide a sufficient condition that the cokernels $\cok(A)$ and $\cok(A+B_n)$ become independent as $n \rightarrow \infty$, where $B_n$ is a fixed $n \times n$ matrix over $\Z_p$ for each $n$ and $A$ is a random $n \times n$ matrix over $\Z_p$.
\end{abstract}

%---------------------------------------------------------------
%---------------------------------------------------------------
\section{Introduction} \label{Sec1}

The study of the distribution of the cokernel of a random $p$-adic matrix for a prime $p$ was initiated by the work of Washington \cite{Was86} on the Cohen-Lenstra heuristics. Based on this work, Friedman and Washington \cite{FW89} proved that for a finite abelian $p$-group $H$, 
\begin{equation} \label{eq1a}
\lim_{n \rightarrow \infty} \underset{A \in \M_{n}(\Z_p)}{\Prob} (\cok (A) \cong H) = \frac{1}{\left | \Aut(H) \right |} \prod_{i=1}^{\infty}(1-p^{-i})
\end{equation}
where $\M_n(\Z_p)$ denotes the set of $n \times n$ matrices over $\Z_p$ and $A$ is random with respect to the Haar probability measure on $\M_n(\Z_p)$. For an odd $p$, the right hand side of the formula (\ref{eq1a}) coincides with the conjectured distribution of the $p$-parts of the class groups of imaginary quadratic fields predicted by Cohen and Lenstra \cite{CL84}. When $p=2$, this coincides with the distribution of $(2 \Cl(K))[2^{\infty}]$ ($\Cl(K)$ denotes the class group of $K$) for imaginary quadratic fields $K$, as conjectured by Gerth \cite{Ger87} and proved by Smith \cite[Theorem 1.9]{Smi22}.

Recently, there have been a lot of works which concern the distribution of the cokernels of various types of random matrices over $\Z_p$. 
The distribution of the cokernel of a random symmetric (resp. alternating) matrix over $\Z_p$ was determined by Clancy, Kaplan, Leake, Payne and Wood \cite{CKLPW15} (resp. Bhargava, Kane, Lenstra, Poonen and Rains \cite{BKLPR15}). 
A model of the $p$-parts of the Tate-Shafarevich groups of elliptic curves over $\Q$ in terms of the cokernels of random alternating matrices was also suggested in \cite{BKLPR15}. 
The distribution of this random matrix model is compatible with the prediction of Delaunay (\cite{Del01}, \cite{Del07}, \cite{DJ14}) on the distribution of the $p$-parts of the Tate-Shafarevich groups.
Wood \cite{Woo19} extended the formula (\ref{eq1a}) to a large family of random matrices in $\M_n(\Z_p)$ whose entries are independent and their reductions modulo $p$ are not too concentrated. 
The key step of the proof of such universality result is to determine the moments of the cokernels of random matrices. See also \cite{NW22}, \cite{Woo17} for more works in this direction.

Friedman and Washington also proved that
\begin{equation} \label{eq1b}
\lim_{n \rightarrow \infty} \underset{A \in \GL_{n}(\Z_p)}{\Prob} (\cok (A-I_n) \cong H) = \lim_{n \rightarrow \infty} \underset{A \in \M_{n}(\Z_p)}{\Prob} (\cok (A) \cong H)
\end{equation}
($I=I_n$ denotes the $n \times n$ identity matrix) for every finite abelian $p$-group $H$. This means that as $A$ ranges over $\M_n(\Z_p)$, the events $\cok(A) = 0$ and $\cok(A-I_n) \cong H$ become independent as $n \rightarrow \infty$. 
As a natural generalization, Cheong and Huang \cite{CH21} suggested a conjecture on the joint distribution of $\cok(P_1(A)), \cdots, \cok(P_l(A))$ where $P_1(t), \cdots, P_l(t) \in \Z_p[t]$ are monic polynomials whose reductions modulo $p$ in $\F_p[t]$ are distinct and irreducible, and $A \in \M_n(\Z_p)$ is a random matrix. 
Cheong and Kaplan \cite{CK22} modified the conjecture and proved the following proposition by counting the number of matrices with given cokernel conditions and a given reduction modulo $p$.

\begin{proposition} \label{prop11}
(\cite[Theorem 1.1]{CK22}) Let $P_1(t), \cdots, P_l(t) \in \Z_p[t]$ be monic polynomials whose reductions modulo $p$ in $\F_p[t]$ are distinct and irreducible, and let $H_j$ be a finite module over $R_j := \Z_p[t]/(P_j(t))$ for each $1 \leq j \leq l$. Assume that $\deg (P_j) \leq 2$ for each $j$. Then we have
\begin{equation} \label{eq1c}
\lim_{n \rightarrow \infty} \underset{A \in \M_{n}(\Z_p)}{\Prob} \begin{pmatrix}
\cok(P_j(A)) \cong H_j \\ 
\text{ for } 1 \leq j \leq l
\end{pmatrix} 
= \prod_{j=1}^{l} \left ( \frac{1}{\left | \Aut_{R_j}(H_j) \right |} \prod_{i=1}^{\infty}(1-p^{-i \deg (P_j)}) \right ).
\end{equation}
\end{proposition}
The above proposition implies that the distributions of the cokernels $\cok(P_j(A))$ ($1 \leq j \leq l$) become independent as $n \rightarrow \infty$. To clarify the above proposition, we note the following. 

\begin{itemize}
    \item Each $R_j$ is a complete discrete valuation ring with a finite residue field $\F_p[t]/(\overline{P_j}(t)) \cong \F_{p^{\deg (P_j)} }$. This follows from the fact that the reduction modulo $p$ of each $P_j(t)$ is irreducible.
    
    \item The cokernel $\cok(P_j(A))$ has a natural $R_j$-module structure defined by $t \cdot x := Ax$. 
\end{itemize}

The first main theorem of the paper shows that the above proposition holds without any restriction on the degree of the polynomials. Our proof is based on probabilistic arguments using Lemma \ref{lem21b}, which is different from the proof of Proposition \ref{prop11} relies on the reduction modulo $p$. Its proof is given in Section \ref{Sub22} and the lemmas that will be used in the proof are summarized in Section \ref{Sub21}. 

\begin{theorem} \label{thm12}
(Theorem \ref{mainthm1}) Proposition \ref{prop11} is true without the assumption $\deg (P_j) \leq 2$.
\end{theorem}

The relations between Theorem \ref{thm12} and the linearization of a random matrix model introduced by Lipnowski, Sawin and Tsimerman \cite{LST20} are discussed in Section \ref{Sec3}. 
In Section \ref{Sub31}, first we briefly explain the idea of the linearization.
After that, we provide a heuristic argument based on Theorem \ref{thm12} which suggests one possible way to understand a surprising result of \cite{LST20}: a non-linear random matrix model and its linearization have the same distribution.

Theorem \ref{thm12} is no longer true when there are two polynomials $P_{j_1}(t)$ and $P_{j_2}(t)$ with the same reduction modulo $p$. 
In this case, two cokernels $\cok(P_{j_1}(A))$ and $\cok(P_{j_2}(A))$ have the same $p$-rank so they can never be independent.
Nevertheless, it is still possible to consider the joint distribution of the cokernels $\cok(P_j(A))$ ($1 \leq j \leq l$).
Section \ref{Sub32} contains a conjectural generalization of Theorem \ref{thm12} (i.e. Conjecture \ref{conj32b}) for such polynomials, which is inspired by a heuristic argument based on the linearization.

Proposition \ref{prop11} (or Theorem \ref{thm12}) implies that for a random matrix $A \in \M_n(\Z_p)$, the cokernels $\cok(A)$ and $\cok(A+I_n)$ become independent as $n \rightarrow \infty$. 
One may ask under which conditions the cokernels $\cok(A)$ and $\cok(A+B_n)$ become independent, where $\left \{ B_n \right \}_{n \geq 1}$ is a sequence of matrices such that $B_n \in \M_n(\Z_p)$.
The second main theorem states that such independence phenomenon holds if the $p$-rank of $\cok(B_n)$ is not too close to $n$ (equivalently, the rank of $\overline{B_n} \in \M_n(\F_p)$ is not too small). 
For a finitely generated $\Z_p$-module $M$, denote its $p$-rank by $r_p(M) := \rank_{\F_p}(M /pM)$. 
\begin{theorem} \label{thm13}
(Theorem \ref{mainthm2}) Let $\left \{ B_n \right \}_{n \geq 1}$ be any sequence of matrices such that $B_n \in \M_n(\Z_p)$ and
\begin{equation} \label{eq1d}
\lim_{n \rightarrow \infty} (n - \log_p n - r_p(\cok(B_n)) ) = \infty.
\end{equation}
Then for every finite abelian $p$-groups $H_1$ and $H_2$, we have
\begin{equation} \label{eq1e}
\lim_{n \rightarrow \infty} \underset{A \in \M_{n}(\Z_p)}{\Prob}
\begin{pmatrix}
\cok(A) \cong H_1 \text{ and } \\ 
\cok(A+B_n) \cong H_2
\end{pmatrix} = \prod_{j=1}^{2} \left ( \frac{1}{\left | \Aut_{\Z_p}(H_j) \right |} \prod_{i=1}^{\infty}(1-p^{-i}) \right ).
\end{equation}
\end{theorem}

We hope to extend our results to various types of random matrices over $\Z_p$. For example, one may calculate the distribution of $\cok(A^2+I_n)$ or the joint distribution of $\cok(A)$ and $\cok(A+I_n)$ for a random symmetric $n \times n$ matrix $A$ over $\Z_p$. 
Also it would be interesting to find more connections between our results and the linearization of a random matrix model, starting from the discussions in Section \ref{Sec3}.

%---------------------------------------------------------
%---------------------------------------------------------
\section{Joint distribution of the cokernels of multiple polynomials of a random $p$-adic matrix} \label{Sec2}

The purpose of this section is to prove the following theorem.

\begin{theorem} \label{mainthm1}
Let $P_1(t), \cdots, P_l(t) \in \Z_p[t]$ be monic polynomials whose reductions modulo $p$ in $\F_p[t]$ are distinct and irreducible, and let $H_j$ be a finite module over $R_j := \Z_p[t]/(P_j(t))$ for each $1 \leq j \leq l$. Then we have
\begin{equation} \label{eq20a}
\lim_{n \rightarrow \infty} \underset{A \in \M_{n}(\Z_p)}{\Prob} \begin{pmatrix}
\cok(P_j(A)) \cong H_j \\ 
\text{ for } 1 \leq j \leq l
\end{pmatrix} 
= \prod_{j=1}^{l} \left ( \frac{1}{\left | \Aut_{R_j}(H_j) \right |} \prod_{i=1}^{\infty}(1-p^{-i \deg (P_j)}) \right ).
\end{equation}
\end{theorem}

%---------------------------------------------------------
\subsection{Some lemmas} \label{Sub21}

Before starting the proof, we provide some lemmas that will be useful for the proof. The first lemma provides a linear description of the cokernel of a matrix polynomial. 

%---------------------------------------------------------------
\begin{lemma} \label{lem21a}
(cf. \cite[Lemma 3.2]{CK22}) For each $j$, the map 
$$
\Z_p^n/P_j(A) \Z_p^n = \cok(P_j(A)) \rightarrow \cok_{R_j}(A-tI) := R_j^n / (A-tI)R_j^n \;\; (\overline{x} \mapsto \overline{x})
$$
($A \in \M_n(\Z_p)$, $A-tI \in \M_n(R_j)$) is an $R_j$-module isomorphism.
\end{lemma}

For $A \in \M_n(\Z_p)$ and $P, Q \in \GL_n(\Z_p)$, we have $\cok(A) \cong \cok(PAQ)$ so the elementary operations do not change the cokernel of a matrix. The key tool in the proof of Theorem \ref{mainthm1} is the reduction of the size of a matrix by elementary operations without affecting the cokernels. The next lemma will be repeatedly used in this procedure. 

%---------------------------------------------------------------
\begin{lemma} \label{lem21b}
For any integers $n \geq r > 0$ and a Haar-random matrix $C \in \M_{n \times r}(\Z_p)$, we have
$$
\underset{C \in \M_{n \times r}(\Z_p)}{\Prob} \left ( \text{there exists } Y \in \GL_{n}(\Z_p) \text{ such that } YC = \begin{pmatrix}
I_r \\
O
\end{pmatrix} \right ) = c_{n,r} := \prod_{j=0}^{r-1} \left ( 1 - \frac{1}{p^{n-j}} \right ).
$$
\end{lemma}

\begin{proof}
For a random $C \in \M_{n \times r}(\Z_p)$, its modulo $p$ reduction $C_0$ is a random element of $\M_{n \times r}(\F_p)$. 
Since the probability that $\rank_{\F_p}(C_0) = r$ is given by $c_{n,r}$, it is enough to show that we can always find $Y \in \GL_n(\Z_p)$ such that $YC = \begin{pmatrix}
I_r \\
O
\end{pmatrix}$ when $\rank_{\F_p}(C_0) = r$. If $C_0$ has the full rank, there exists $Y_0 \in \GL_n(\F_p)$ such that $Y_0C_0 = \begin{pmatrix}
I_r \\
O
\end{pmatrix}$. Choose any matrix $Y_1 \in \GL_n(\Z_p)$ whose reduction modulo $p$ is $Y_0$ and write $Y_1C = \begin{pmatrix}
W_1 \\
W_2
\end{pmatrix} \in \M_{(r+(n-r)) \times r}(\Z_p)$. Since $W_1 \equiv I_r \;\; (\text{mod } p)$, the determinant of $W_1$ is unit in $\Z_p$ so $W_1 \in \GL_r(\Z_p)$. Now we have
  $$\left(\begin{array}{@{}c|c@{}}
    W_1^{-1} & O \\\hline
    -W_2W_1^{-1} & I_{n-r}
  \end{array}\right) Y_1 C = \begin{pmatrix}
I_r \\
O
\end{pmatrix},
  $$
which finishes the proof.
\end{proof}

%---------------------------------------------------------------
\begin{lemma} \label{lem21c}
For $R_0 := \Z_p[t]/(P_1(t) \cdots P_l(t))$, the map
$$
\M_n(R_0) \rightarrow \prod_{j=1}^{l} \M_n(R_j) \;\; (A \mapsto (\overline{A}, \cdots, \overline{A}))
$$
is a ring isomorphism.
\end{lemma}

\begin{proof}
Since $\overline{P_i}$ and $\overline{P_j}$ have no common factor, the resultant of $P_i$ and $P_j$ is a unit in $\Z_p$. The resultant of $P_i$ and $P_j$ is a $\Z_p[t]$-linear combination of $P_i$ and $P_j$ so there are $c_i, c_j \in \Z_p[t]$ such that $c_iP_i+c_jP_j=1$. Since the ideals $(P_1), \cdots, (P_l)$ in $\Z_p[t]$ are pairwise coprime, the map
$$
R_0 \rightarrow \prod_{j=1}^{l} R_j \;\; (x \mapsto (x+(P_1), \cdots, x+(P_l)))
$$
is a ring isomorphism by the Chinese remainder theorem. Therefore the map $\displaystyle \M_n(R_0) \rightarrow \prod_{j=1}^{l} \M_n(R_j)$ is also a ring isomorphism. 
\end{proof}

%---------------------------------------------------------
%---------------------------------------------------------
\subsection{Proof of Theorem \ref{mainthm1}} \label{Sub22}

By Lemma \ref{lem21a}, Theorem \ref{mainthm1} is equivalent to 
\begin{equation} \label{eq22a}
\lim_{n \rightarrow \infty} \underset{A \in \M_{n}(\Z_p)}{\Prob} \begin{pmatrix}
\cok_{R_j}(A-tI) \cong H_j \\ 
\text{ for } 1 \leq j \leq l
\end{pmatrix} 
= \prod_{j=1}^{l} \left ( \frac{1}{\left | \Aut_{R_j}(H_j) \right |} \prod_{i=1}^{\infty}(1-p^{-i \deg (P_j)}) \right ).
\end{equation}
For $M \in \M_n(\Z_p[t])$, denote its image in $\M_n(R_j)$ by $M(j)$. For any $n>r>0$ and $k \geq 0$, denote
$$
P_{n, r}^{k} := \underset{\substack{A \in \M_n(\Z_p) \\ B_1, \cdots, B_k \in \M_{r \times n}(\Z_p)}}{\Prob} \begin{pmatrix}
\cok_{R_j}(M_{A, \, [B_1, \cdots, B_k]}(j)) \cong H_j \\ 
\text{ for } 1 \leq j \leq l
\end{pmatrix} 
$$
for
$$
M_{A, \, [B_1, \cdots, B_k]} := A + \begin{pmatrix}
\sum_{i=1}^{k} t^iB_i \\ 
O_{(n-r) \times n}
\end{pmatrix}
- \begin{pmatrix}
t^{k+1} I_r & O \\ 
O & t I_{n-r}
\end{pmatrix} \in \M_n(\Z_p[t]).
$$
It is obvious that
\begin{equation} \label{eq22b}
P_{n, r}^{0} = \underset{A \in \M_{n}(\Z_p)}{\Prob} \begin{pmatrix}
\cok_{R_j}(A-tI) \cong H_j \\ 
\text{ for } 1 \leq j \leq l
\end{pmatrix}. 
\end{equation}

%---------------------------------------------------------------
\noindent For $n > 2r > 0$, $A = \begin{pmatrix}
A_1 & A_2 \\ 
A_3 & A_4
\end{pmatrix} \in \M_{r+(n-r)}(\Z_p)$, $B_1, \cdots, B_k \in \M_{r \times n}(\Z_p)$ and 
  $$
  U = \begin{pmatrix}
    I_r & O \\
    O & U_1
  \end{pmatrix} \in \GL_{n}(\Z_p) \;\; (U_1 \in \GL_{n-r}(\Z_p)),
  $$
we have
\begin{equation*}
U M_{A, \, [B_1, \cdots, B_k]} U^{-1} 
= M_{ A', \,  [B_1', \cdots, B_k']}
\end{equation*}
for
\begin{equation*}
A' = \begin{pmatrix}
    A_1 & A_2U_1^{-1} \\
    U_1A_3 & U_1A_4U_1^{-1}
  \end{pmatrix}, \, B_i'=B_i U^{-1}
\end{equation*}
by a direct computation. 
For a random $A_3$, the probability that there exists $U_1 \in \GL_{n-r}(\Z_p)$ such that $U_1A_3 = \begin{pmatrix}
I_r \\
O
\end{pmatrix}$ is at least $c_{n-r, r}$ by Lemma \ref{lem21b}. Moreover, for any given $A_3$ and $U_1$, the matrices $A_2U_1^{-1}$, $U_1A_4U_1^{-1}$ and $B_iU^{-1}$ ($1 \leq i \leq k$) are random and independent if and only if the matrices $A_2$, $A_4$ and $B_i$ are random and independent. 
These imply that
\begin{equation} \label{eq22c}
\left | P_{n, r}^{k} - \widetilde{P}_{n, r}^{k} \right | \leq (1-c_{n-r, r})
\end{equation}
for
$$
\widetilde{\M}_{n, r}(\Z_p) := \left \{ \begin{pmatrix}
A_1 & A_2 & A_3\\ 
I_r & A_4 & A_5\\ 
O & A_6 & A_7
\end{pmatrix} \in \M_{r+r+(n-2r)}(\Z_p) \right \} \subset \M_n(\Z_p)$$
and
$$
\widetilde{P}_{n, r}^{k} := \underset{\substack{A \in \widetilde{\M}_{n, r}(\Z_p) \\ B_1, \cdots, B_k \in \M_{r \times n}(\Z_p)}}{\Prob} \begin{pmatrix}
\cok_{R_j}(M_{A, \, [B_1, \cdots, B_k]}(j)) \cong H_j \\ 
\text{ for } 1 \leq j \leq l
\end{pmatrix} .
$$
%---------------------------------------------------------------
Let
$$
A = \begin{pmatrix}
A_1 & A_2 & A_3 \\ 
I_r & A_4 & A_5 \\ 
O & A_6 & A_7 
\end{pmatrix} \in \widetilde{\M}_{n, r}(\Z_p)
$$
and 
$$
B_i = \begin{pmatrix}
X_i & Y_i & Z_i
\end{pmatrix} \in \M_{r \times (r+r+(n-2r))}(\Z_p)
$$
for $1 \leq i \leq k$. By elementary operations, we have
\begin{equation*}
\small
\begin{split}
& \; M_{A, \, [B_1, \cdots, B_k]} \\
= & \, \left(\begin{array}{@{}c|c|c@{}}
    A_1 + \sum_{i=1}^{k} t^iX_i - t^{k+1}I_r & A_2 + \sum_{i=1}^{k} t^iY_i & A_3 + \sum_{i=1}^{k} t^iZ_i \\\hline
I_r & A_4-tI_r & A_5 \\\hline
    O & A_6 & A_7-tI_{n-2r} 
  \end{array}\right) \\
\Rightarrow & \, \left(\begin{array}{@{}c|c|c@{}}
    O & (A_2 + \sum_{i=1}^{k} t^iY_i)-(A_1 + \sum_{i=1}^{k} t^iX_i - t^{k+1}I_r)(A_4-tI_r) & (A_3 + \sum_{i=1}^{k} t^iZ_i)-(A_1 + \sum_{i=1}^{k} t^iX_i - t^{k+1}I_r)A_5 \\\hline
    I_r & A_4-tI_r & A_5 \\\hline
    O & A_6 & A_7-tI_{n-2r} 
  \end{array}\right) \\
\Rightarrow & \, \left(\begin{array}{@{}c|c@{}}
    (A_2 + \sum_{i=1}^{k} t^iY_i)-(A_1 + \sum_{i=1}^{k} t^iX_i - t^{k+1}I_r)(A_4-tI_r) & (A_3 + \sum_{i=1}^{k} t^iZ_i)-(A_1 + \sum_{i=1}^{k} t^iX_i - t^{k+1}I_r)A_5 \\ \hline
    A_6 & A_7-tI_{n-2r} 
  \end{array}\right) \\
= & \, \left(\begin{array}{@{}c|c@{}}
    A_2-A_1A_4 & A_3-A_1A_5 \\ \hline
    A_6 & A_7
  \end{array}\right) \\
  & + \left(\begin{array}{@{}c|c@{}}
  t(Y_1-X_1A_4+A_1)
  + \sum_{i=2}^{k} t^i(Y_i-X_iA_4+X_{i-1})
  + t^{k+1}(A_4+X_k)
  & \sum_{i=1}^{k} t^i(Z_i-X_iA_5) + t^{k+1}A_5 \\ \hline
  O & O
  \end{array}\right) \\
  & - \left(\begin{array}{@{}c|c@{}}
    t^{k+2}I_r & O \\ \hline
    O & tI_{n-2r} 
  \end{array}\right) \\
  =: & \, M_{A', \, [B_1', \cdots, B_{k+1}']}.
\end{split}
\end{equation*}
Since the elementary operations do not change the cokernel, we have
$$
\cok(M_{A, \, [B_1, \cdots, B_k]}(j)) \cong \cok(M_{A', \, [B_1', \cdots, B_{k+1}']}(j))
$$
for all $j$. 
%----------
The matrices $B_1', \cdots, B_{k}'$ are given by $B_i' = \begin{pmatrix}
Y_i & Z_i
\end{pmatrix} + N_i$ ($1 \leq i \leq k$) for some $N_1, \cdots, N_{k} \in \M_{r \times (n-r)}(\Z_p)$ depending only on $A_1$, $A_4$, $A_5$ and $X_i$ ($1 \leq i \leq k$). Similarly, $\displaystyle A' = \begin{pmatrix}
A_2 & A_3 \\
A_6 & A_7 \\
\end{pmatrix} + N$ for some $N \in \M_{n-r}(\Z_p)$ depending only on $A_1$, $A_4$ and $A_5$ and we also have $B_{k+1}' = \begin{pmatrix}
A_4 & A_5
\end{pmatrix} + \begin{pmatrix}
X_k & O
\end{pmatrix}$. Therefore $A', B_1', \cdots, B_{k+1}'$ are random and independent if $Y_i$, $Z_i$ ($1 \leq i \leq k$), $A_j$ ($2 \leq j \leq 7$) are random and independent, or $A, B_1, \cdots, B_k$ are random and independent. This implies that
\begin{equation} \label{eq22d}
\widetilde{P}_{n, r}^{k} = P_{n-r, r}^{k+1}.
\end{equation}
%----------
By the equations (\ref{eq22c}) and (\ref{eq22d}), we have
\begin{equation} \label{eq22e}
\begin{split}
\left | P_{n,r}^{0} - P_{n-(d-1)r,r}^{d-1} \right |
& \leq \sum_{i=1}^{d-1} \left | P_{n-(i-1)r,r}^{i-1} - P_{n-ir,r}^{i} \right | \\
& \leq \sum_{i=1}^{d-1} (1 - c_{n-ir, r})
\end{split}
\end{equation}
for $d := \deg (P_1) + \cdots + \deg (P_l)$ and $n>dr$. 

%---------------------------------------------------------------
Now let $r>s>0$ and consider the probability $P_{r+s, r}^{d-1}$. For $A = \begin{pmatrix}
A_1 & A_2 \\
    A_3 & A_4
\end{pmatrix} \in \M_{r+s}(\Z_p)$, $B_i = \begin{pmatrix}
X_i & Y_i
\end{pmatrix} \in \M_{r \times (r+s)}(\Z_p)$ and 
  $$
  U = \begin{pmatrix}
    U_2 & O \\
    O & I_{s}
  \end{pmatrix} \in \GL_{r+s}(\Z_p) \;\; (U_2 \in \GL_{r}(\Z_p)),
  $$
we have
\begin{equation*}
U M_{A, \, [B_1, \cdots, B_{d-1}]} U^{-1} 
= M_{ A', \,  [B_1', \cdots, B_{d-1}']}
\end{equation*}
for
\begin{equation*}
A' = \begin{pmatrix}
    U_2A_1U_2^{-1} & U_2A_2 \\
    A_3U_2^{-1} & A_4
  \end{pmatrix}, \, B_i'= \begin{pmatrix}
U_2X_iU_2^{-1} & U_2Y_i
\end{pmatrix}.
\end{equation*}
For a random $A_3$, the probability that there exists $U_2 \in \GL_{r}(\Z_p)$ such that $A_3U_2^{-1} = \begin{pmatrix}
I_{s} & O
\end{pmatrix}$ is at least $c_{r, s}$ by Lemma \ref{lem21b}. Moreover, for any given $A_3$ and $U_2$, the matrices $U_2A_1U_2^{-1}$, $U_2A_2$, $U_2X_iU_2^{-1}$ and $U_2Y_i$ ($1 \leq i \leq d-1$) are random and independent if and only if the matrices $A_1$, $A_2$, $X_i$ and $Y_i$ are random and independent. These imply that
\begin{equation} \label{eq22f}
\left | P_{r+s, r}^{d-1} - \widehat{P}_{r+s, s}^{d-1} \right | \leq 1-c_{r, s}
\end{equation}
for
$$
\widehat{\M}_{r+s, s}(\Z_p) := \left \{ \begin{pmatrix}
A_1 & A_2 & A_3\\ 
I_{s} & O & A_4
\end{pmatrix} \in \M_{(r+s) \times (s+(r-s)+s)}(\Z_p) \right \} \subset \M_{r+s}(\Z_p)$$
and
$$
\widehat{P}_{r+s, s}^{d-1} := \underset{\substack{A \in \widehat{\M}_{r+s, s}(\Z_p) \\ B_1, \cdots, B_{d-1} \in \M_{r \times (r+s)}(\Z_p)}}{\Prob} \begin{pmatrix}
\cok_{R_j}(M_{A, \, [B_1, \cdots, B_{d-1}]}(j)) \cong H_j \\ 
\text{ for } 1 \leq j \leq l
\end{pmatrix}.
$$
Let
$$
A = \begin{pmatrix}
A_1 & A_2 & A_3\\ 
A_4 & A_5 & A_6\\
I_{s} & O & A_7
\end{pmatrix} \in \widehat{\M}_{r+s, s}(\Z_p) \subset \M_{s+(r-s)+s}(\Z_p)
$$
and
$$
B_i = \begin{pmatrix}
B_{i, 1} & B_{i, 2} & B_{i, 3}\\
B_{i, 4} & B_{i, 5} & B_{i, 6}
\end{pmatrix} \in \M_{(s+(r-s)) \times (s+(r-s)+s)}(\Z_p)
$$ 
for $1 \leq i \leq d-1$. By elementary operations, we have
\begin{equation*}
\small
\begin{split}
& \; M_{A, \, [B_1, \cdots, B_{d-1}]} \\
= & \, \left(\begin{array}{@{}c|c|c@{}}
    A_1 + \sum_{i=1}^{d-1} t^iB_{i, 1} - t^{d}I_s & A_2 + \sum_{i=1}^{d-1} t^iB_{i, 2} & A_3 + \sum_{i=1}^{d-1} t^iB_{i, 3} \\\hline
A_4 + \sum_{i=1}^{d-1} t^iB_{i, 4} & A_5 + \sum_{i=1}^{d-1} t^iB_{i, 5} -t^{d}I_{r-s} & A_6+ \sum_{i=1}^{d-1} t^iB_{i, 6} \\\hline
    I_s & O & A_7-tI_{s} 
  \end{array}\right) \\
\Rightarrow & \, \left(\begin{array}{@{}c|c|c@{}}
    A_1 + \sum_{i=1}^{d-1} t^iB_{i, 1} - t^{d}I_s & A_2 + \sum_{i=1}^{d-1} t^iB_{i, 2} & (A_3 + \sum_{i=1}^{d-1} t^iB_{i, 3}) - (A_1 + \sum_{i=1}^{d-1} t^iB_{i, 1} - t^{d}I_s)(A_7-tI_{s}) \\\hline
A_4 + \sum_{i=1}^{d-1} t^iB_{i, 4} & A_5 + \sum_{i=1}^{d-1} t^iB_{i, 5} -t^{d}I_{r-s} & (A_6 + \sum_{i=1}^{d-1} t^iB_{i, 6}) - (A_4 + \sum_{i=1}^{d-1} t^iB_{i, 4})(A_7-tI_{s}) \\\hline
    I_s & O & O
  \end{array}\right) \\
\Rightarrow & \, \left(\begin{array}{@{}c|c@{}}
    A_2 + \sum_{i=1}^{d-1} t^iB_{i, 2} & (A_3 + \sum_{i=1}^{d-1} t^iB_{i, 3}) - (A_1 + \sum_{i=1}^{d-1} t^iB_{i, 1} - t^{d}I_s)(A_7-tI_{s}) \\\hline
A_5 + \sum_{i=1}^{d-1} t^iB_{i, 5} -t^{d}I_{r-s} & (A_6 + \sum_{i=1}^{d-1} t^iB_{i, 6}) - (A_4 + \sum_{i=1}^{d-1} t^iB_{i, 4})(A_7-tI_{s})
  \end{array}\right) \\
= & \, \left(\begin{array}{@{}c|c@{}}
    A_2 & A_3-A_1A_7 \\ \hline
    A_5 & A_6-A_4A_7
  \end{array}\right) + t \left(\begin{array}{@{}c|c@{}}
    B_{1,2} & B_{1, 3}+A_1-B_{1,1}A_7 \\ \hline
    B_{1,5} & B_{1, 6}+A_4-B_{1,4}A_7
  \end{array}\right) \\
  & + \sum_{i=2}^{d-1} t^i \left(\begin{array}{@{}c|c@{}}
  B_{i, 2} & B_{i,3}-B_{i,1}A_7+B_{i-1,1} \\ \hline
  B_{i, 5} & B_{i,6}-B_{i,4}A_7+B_{i-1,4}
  \end{array}\right) + \left(\begin{array}{@{}c|c@{}}
    O & t^d(B_{d-1, 1}+A_7)-t^{d+1}I_s \\ \hline
    -t^d I_{r-s} & t^dB_{d-1, 4}
  \end{array}\right) \\
  = & \, \sum_{i=0}^{d-1}t^iC_i + t^d \begin{pmatrix} O & E
\end{pmatrix} 
+ \begin{pmatrix}  
O & -t^{d+1}I_s \\
-t^d I_{r-s} & O
\end{pmatrix} \\
=: & \, N_{[C_0, \cdots, C_{d-1}], E} \in \M_r(\Z_p[t])
\end{split}
\end{equation*}
for $C_0, \cdots, C_{d-1} \in \M_{r}(\Z_p)$ and $E \in \M_{r \times s}(\Z_p)$. 
For given $A_1$, $A_4$, $A_7$, $B_{i,1}$ and $B_{i,4}$ ($1 \leq i \leq d-2$), the matrices $C_0, \cdots, C_{d-1}, E$ are random and independent if $A_j$, $B_{i, j}$ ($1 \leq i \leq d-1$, $j \in \left\{ 2, 3, 5, 6 \right\}$), $B_{d-1, 1}$ and $B_{d-1, 4}$ are random and independent. 
In this case, the image of $N_{[C_0, \cdots, C_{d-1}], E}$ in $\M_r(R_0)$ is also random so Lemma \ref{lem21c} implies that 
$$
N_{[C_0, \cdots, C_{d-1}], E}(j) \in \M_r(R_j) \;\; (1 \leq j \leq l)
$$
are random and independent. Therefore
\begin{equation} \label{eq22g}
\begin{split}
\widehat{P}_{r+s, s}^{d-1} 
& = \underset{\substack{C_0, \cdots, C_{d-1} \in \M_{r}(\Z_p) \\ E \in \M_{r \times s}(\Z_p)}}{\Prob} \begin{pmatrix}
\cok_{R_j}(N_{[C_0, \cdots, C_{d-1}], E}(j)) \cong H_j \\ 
\text{ for } 1 \leq j \leq l 
\end{pmatrix} \\
& = \prod_{j=1}^{l} \left ( \underset{N_j \in \M_{r}(R_j)}{\Prob} (\cok(N_j) \cong H_j)  \right ).
\end{split}
\end{equation}
(Each $N_j$ is a Haar-random matrix in $\M_{r}(R_j)$.) By the equations (\ref{eq22f}) and (\ref{eq22g}), we have
\begin{equation} \label{eq22h}
\left | P_{r+s, r}^{d-1} - \prod_{j=1}^{l} \left ( \underset{N_j \in \M_{r}(R_j)}{\Prob} (\cok(N_j) \cong H_j)  \right ) \right | \leq 1-c_{r, s}.
\end{equation}

%---------------------------------------------------------------
Now let $(t_n)_{n \geq 1}$ be any sequence of non-negative integers such that 
$$
\lim_{n \rightarrow \infty} (n-dt_n) = \lim_{n \rightarrow \infty} ((d+1)t_n-n) = \infty
$$
and denote $s_n := n - dt_n$. (For example, we can choose $\displaystyle t_n = \left \lfloor \frac{n}{d+1/2} \right \rfloor$.) Assume that $n$ is sufficiently large so that both $s_n$ and $t_n-s_n$ are positive. By the equations (\ref{eq22b}), (\ref{eq22e}) and (\ref{eq22h}), we have
\begin{equation*}
\begin{split}
& \left | \underset{A \in \M_{n}(\Z_p)}{\Prob} \begin{pmatrix}
\cok_{R_j}(A-tI) \cong H_j \\ 
\text{ for } 1 \leq j \leq l
\end{pmatrix} - \prod_{j=1}^{l} \left ( \underset{N_j \in \M_{t_n}(R_j)}{\Prob} (\cok(N_j) \cong H_j)  \right ) \right | \\
\leq & \, \sum_{i=1}^{d-1} (1 - c_{n-it_n, t_n}) + (1-c_{t_n, s_n}).
\end{split}
\end{equation*}
Here we take $r=t_n$ in the equation (\ref{eq22e}) and $(r,s)=(t_n, s_n)$ in the equation (\ref{eq22h}). 
Since both $t_n-s_n$ and $s_n$ go to infinity as $n \rightarrow \infty$, we have
$$
\lim_{n \rightarrow \infty} (\sum_{i=1}^{d-1} (1 - c_{n-it_n, t_n}) + (1-c_{t_n, s_n})) = 0
$$
so
\begin{equation*}
\begin{split}
\lim_{n \rightarrow \infty} \underset{A \in \M_{n}(\Z_p)}{\Prob} \begin{pmatrix}
\cok_{R_j}(A-tI) \cong H_j \\ 
\text{ for } 1 \leq j \leq l
\end{pmatrix}
& = \lim_{n \rightarrow \infty} \prod_{j=1}^{l} \left ( \underset{N_j \in \M_{t_n}(R_j)}{\Prob} (\cok(N_j) \cong H_j)  \right ) \\
& = \prod_{j=1}^{l} \left ( \frac{1}{\left | \Aut_{R_j}(H_j) \right |} \prod_{i=1}^{\infty}(1-p^{-i \deg (P_j)}) \right )
\end{split}
\end{equation*}
by \cite[Proposition 2.1]{CH21}. This finishes the proof of Theorem \ref{mainthm1}. $\square$

%---------------------------------------------------------
%---------------------------------------------------------
\section{Linearization} \label{Sec3}

%---------------------------------------------------------
\subsection{Linearization of a random matrix model} \label{Sub31}

We assume that $p$ is odd in this section.
The original motivation of the formula (\ref{eq1b}) by Friedman and Washington was the function field analogue of the Cohen-Lenstra conjecture. 
They suggested that the distribution of $\cok(I-F)$ for a random matrix $F \in \GSp_{2g}^{(q)}(\Z_p)$ converges to the distribution of the $p$-parts of the class groups of imaginary (i.e. ramified at $\infty$) quadratic extensions of $\F_q(t)$ as $g \rightarrow \infty$. Here $\GSp_{2g}^{(q)}(\Z_p)$ denotes the subset of $\GSp_{2g}(\Z_p)$ consisting of symplectic similitudes of similitude factor $q$. Instead of calculating the distribution of $\cok(I-F)$ for $F \in \GSp_{2g}^{(q)}(\Z_p)$, they calculated the distribution of $\cok(I-F)$ for $F \in \GL_{2g}(\Z_p)$. 

Achter \cite{Ach06} showed that the distribution of $\cok(I-F)$ for $F \in \GSp_{2g}^{(q)}(\Z_p)$ does not agree with the distribution of $\cok(I-F)$ for $F \in \M_{2g}(\Z_p)$, which disproves the conjecture of Friedman and Washington. It turns out that the distribution of $\cok(I-F)$ for $F \in \GSp_{2g}^{(q)}(\Z_p)$ depending on the exponent of $p$ in $q-1$. Garton \cite[Corollary 5.2.2]{Gar15} calculated the $g \rightarrow \infty$ limit of this distribution in the cases $p^v \parallel  q-1$ for $v \in \left\{ 1, 2 \right\}$. Lipnowski, Sawin and Tsimerman \cite[Theorem 8.11]{LST20} extended this result to every positive integer $v$. (Note that in the statement of \cite[Theorem 8.11]{LST20}, $\mu_g$ should be changed to $\mu$ and $n \rightarrow \infty$ should be changed to $g \rightarrow \infty$. The universal measure $\mu$ appears in \cite[Theorem 8.7]{LST20} depending on $n$.) The linearization of a random matrix model plays a crucial role in their work.

We introduce the linearization of a random matrix model following the exposition in \cite{LST20}. Assume that $p^v \parallel q-1$ for some $v>0$. For $F \in \GSp_{2g}^{(q)}(\Z_p)$, consider its logarithm 
$$
\log (F) := \sum_{k=1}^{\infty} \frac{(-1)^{k-1}}{k}(F-I)^k = (F-I) \sum_{k=1}^{\infty} \frac{(-1)^{k-1}}{k}(F-I)^{k-1}.
$$
If $F \equiv I \;\; (\text{mod } p)$, then the above series converges and $\cok(\log (F)) \cong \cok (F-I)$. By the definition of $\GSp_{2g}^{(q)}(\Z_p)$, we have $\displaystyle \frac{F}{q^{\frac{1}{2}}} \in \Sp_{2g}(\Z_p)$ so $\displaystyle \log (F) = M + \frac{1}{2} \log (q) I$ for some $M \in \mathfrak{sp}_{2g}(\Z_p)$. Since $p^v \parallel q-1$ implies that $p^v \parallel \log (q)$, the cokernels $\displaystyle \cok (M + \frac{1}{2} \log (q)I)$ and $\cok(M+p^v I)$ have the same distribution.

Now it is natural to define the linearization of $\cok(F-I)$ ($F \in \GSp_{2g}^{(q)}(\Z_p)$) by $\cok(M+p^v I)$ ($M \in \mathfrak{sp}_{2g}(\Z_p)$). 
By an indirect argument using the moments, Lipnowski, Sawin and Tsimerman proved that $\cok(F-I)$ for a random $F \in \GSp_{2g}^{(q)}(\Z_p)$ and $\cok(M+p^vI)$ for a random $M \in \mathfrak{sp}_{2g}(\Z_p)$ have the same distribution.
If we replace $\GSp_{2g}^{(q)}(\Z_p)$ with $\GL_n(\Z_p)$, the linearization of $\cok(A-I)$ ($A \in \GL_n(\Z_p)$) is given by $\cok(M)$ ($M \in \M_n(\Z_p)$). These two cokernels have the same distribution by the formula (\ref{eq1b}).

One may think that the cokernels $\cok(\log(F))$ and $\cok(F-I)$ have the same distribution because $F-I$ is a good enough approximation of $\log(F)$. In this perspective, the cokernels $\cok (L_1(A)), \, \cok(L_2(A)), \cdots$ should have the same distribution for
$$
L_m(x) = (x-1)E_m(x) := (x-1) \sum_{k=1}^{m} \frac{(-1)^{k-1}}{k}(x-1)^{k-1}.
$$
It seems reasonable because $F \equiv I \;\; (\text{mod } p)$ implies that the cokernels $\cok (L_1(A)), \, \cok(L_2(A)), \cdots$ are the same. However, the following example tells us that this is not true. 

\begin{example} \label{ex31a}
Let $p>3$ be a prime and $A \in \GL_n(\Z_p)$ be a random matrix. 
Since the polynomials $t, t-1, t-3 \in \F_p[t]$ are distinct and irreducible, we have
\begin{equation*}
\begin{split}
& \lim_{n \rightarrow \infty} \underset{A \in \GL_{n}(\Z_p)}{\Prob}
\left ( \cok  ( L_2(A) ) = 0 \right ) \\
= &  \lim_{n \rightarrow \infty} \underset{A \in \GL_{n}(\Z_p)}{\Prob}
(\cok(A-I_n) = \cok(A-3I_n) = 0) \\
= & \prod_{i=1}^{\infty} (1-p^{-i}) \cdot \lim_{n \rightarrow \infty} \underset{A \in \GL_{n}(\Z_p)}{\Prob}
\left ( \cok  ( L_1(A) ) = 0 \right )
\end{split}
\end{equation*}
by Theorem \ref{mainthm1}. (Here we have used the fact that $\cok (A) =0$ if and only if $A \in \GL_n(\Z_p)$.) Therefore the distributions of $\cok(L_2(A))$ and $\cok(L_1(A))$ do not coincide.
\end{example}

Let $\ell$ be a prime. By Eisenstein's criterion, $\ell ! \, E_{\ell}(x+1) \in \Z[x]$ is irreducible over $\Q$ so $E_{\ell}(x)$ is also irreducible over $\Q$. 
Assume that there exists a prime $p = p_{\ell} > \max (\ell, 3)$ such that $\overline{E_{\ell}}(x) \in \F_p[x]$ (reduction modulo $p$ of $E_{\ell}(x) \in \Z_p[x]$) is irreducible. 
The local-global principle does not hold in general for degree $\geq 3$ so the irreducibility of $E_{\ell}(x)$ does not guarantee the existence of such $p$. 
Let $H = H_{\ell}$ be any finite abelian $p$-group whose order is smaller than $p^{\ell-1}$.

For $A \in \GL_n(\Z_p)$, there is a natural surjection 
$$
\cok(L_{\ell}(A)) 
= \cok(E_{\ell}(A)(A-I)) 
\twoheadrightarrow \cok(E_{\ell}(A)).
$$
Since $\overline{E_{\ell}}(x) \in \F_p[x]$ is irreducible, $\cok(E_{\ell}(A))$ has an $R_{\ell} := \Z_p[t]/(E_{\ell}(t))$-module structure. 
The ring $R_{\ell}$ is a complete discrete valuation ring with a finite residue field $\F_p[t]/(\overline{E_{\ell}}(t)) \cong \F_{p^{\ell-1}}$. 
Therefore if $\cok(E_{\ell}(A))$ is nonzero, then it should have at least $p^{\ell-1}$ elements. By the condition $\left | H \right | < p^{\ell-1}$, we have
\begin{equation*}
 \cok(L_{\ell}(A)) \cong H \; \Leftrightarrow \; 
 \cok(E_{\ell}(A))=0  \text{ and } 
 \cok(A-I) \cong \cok(E_{\ell}(A)(A-I)) \cong H.
\end{equation*}
The polynomials $t, t-1, \ell (-1)^{\ell - 1}\overline{E_{\ell}}(t) \in \F_{p_{\ell}}[t]$ are monic, distinct and irreducible so
\begin{equation*}
\begin{split}
& \lim_{n \rightarrow \infty} \underset{A \in \GL_{n}(\Z_{p_{\ell}})}{\Prob}
(\cok(L_{\ell}(A)) \cong H_{\ell}) \\
= & \lim_{n \rightarrow \infty} \underset{A \in \GL_{n}(\Z_{p_{\ell}})}{\Prob}
\begin{pmatrix}
\cok(A-I) \cong H_{\ell} \text{ and } \\ 
\cok(E_{\ell}(A))=0
\end{pmatrix} \\
= & \lim_{n \rightarrow \infty} \underset{A \in \GL_{n}(\Z_{p_{\ell}})}{\Prob}
(\cok(A-I) \cong H_{\ell}) \times \prod_{i=1}^{\infty} (1-p_{\ell}^{-i (\ell -1)})
\end{split}    
\end{equation*}
by Theorem \ref{mainthm1}. 
Taking the limit $\ell \rightarrow \infty$, we obtain
\begin{equation} \label{eq31a}
\lim_{\ell \rightarrow \infty} \lim_{n \rightarrow \infty} \underset{A \in \GL_{n}(\Z_{p_{\ell}})}{\Prob} (\cok(L_{\ell}(A)) \cong H_{\ell})    
= \lim_{\ell \rightarrow \infty} \lim_{n \rightarrow \infty} \underset{A \in \GL_{n}(\Z_{p_{\ell}})}{\Prob}
(\cok(A-I) \cong H_{\ell}).
\end{equation}

This result supports the relation between $\cok(A-I)$ and its linearization, i.e. $\cok(\log A)$ for a random matrix $A$. Although our argument does not give a precise relation between a non-linear random matrix model and its linearization, we hope that it provides some insights on the linearization. 

%---------------------------------------------------------
\subsection{A conjectural generalization of Theorem \ref{mainthm1}} \label{Sub32}

The distinction between the reductions modulo $p$ of the polynomials $P_1(t), \cdots, P_l(t)$ in Theorem \ref{mainthm1} is essential. 
If two monic polynomials $P_1(t), P_2(t) \in \Z_p[t]$ have the same reduction modulo $p$, then $\cok(P_1(A))$ and $\cok(P_2(A))$ have the same $p$-rank so they can never be independent. Even in the simplest case $P_1(t)=t$ and $P_2(t)=t-p^v$ ($v>0$), it is not clear for us how to compute the joint distribution of $\cok(P_1(A))$ and $\cok(P_2(A))$. 

\begin{example} \label{ex32a}
Consider the joint distribution of $\cok(A-I)$ and $\cok(A-(p^v+1)I)$ for a random matrix $A \in \GL_n(\Z_p)$ and a positive integer $v$. If $A \equiv I \;\; (\text{mod } p)$, then
$$
\cok(A-I) \cong \cok(\log A)
$$
and
\begin{equation*}
\begin{split}
\cok(A-(p^v+1)I) 
& \cong \cok(\frac{A}{p^v+1}-I) \\
& \cong \cok(\log A - \log((p^v+1)I)) \\
& = \cok(\log A - \log(p^v+1)I).
\end{split}    
\end{equation*}
Since $p^v \parallel \log(p^v+1)$, the joint distribution of $\cok(M)$ and $\cok(M - \log(p^v+1)I)$ for a random $M \in \M_n(\Z_p)$ is same as the joint distribution of $\cok(M)$ and $\cok(M-p^vI)$. Following the idea of linearization, we expect that for every finite abelian $p$-groups $H_1$ and $H_2$ with $r_p(H_1)=r_p(H_2)$, the distribution of $\cok(A)=0$ and the joint distribution of $\cok(A-I) \cong H_1$ and $\cok(A-(p^v+1)I) \cong H_2$ become independent as $n \rightarrow \infty$, i.e.
$$
\lim_{n \rightarrow \infty} \underset{A \in \GL_{n}(\Z_p)}{\Prob} \begin{pmatrix}
\cok(A-I) \cong H_1 \text{ and } \\ 
\cok(A-(p^v+1)I) \cong H_2
\end{pmatrix} 
= \lim_{n \rightarrow \infty} \underset{M \in \M_{n}(\Z_p)}{\Prob} \begin{pmatrix}
\cok(M) \cong H_1 \text{ and } \\ 
\cok(M-p^vI) \cong H_2
\end{pmatrix}. 
$$
\end{example}

Let $Q_1(t), \cdots, Q_l(t) \in \F_p[t]$ be distinct monic irreducible polynomials and $P_{ij}(t) \in \Z_p[t]$ ($1 \leq j \leq k_i$) be distinct monic polynomials whose reductions modulo $p$ are $Q_i(t)$ for each $1 \leq i \leq l$. For each $i$, let $J_i$ be the joint distribution of the cokernels $\cok(P_{ij}(A))$ ($1 \leq j \leq k_i$). Motivated by the above example, we claim that the distributions $J_1, \cdots, J_l$ become independent as $n \rightarrow \infty$. When $k_i=1$ for all $i$, the following conjecture is a consequence of Theorem \ref{mainthm1}.

\begin{conjecture} \label{conj32b}
Let $H_{ij}$ be a finite module over $R_{ij} := \Z_p[t]/(P_{ij}(t))$ for each $i$ and $j$ such that $r_p(H_{ij}) = r_p(H_{ij'})$ for every $1 \leq i \leq l$ and $1 \leq j, j' \leq k_i$. Then we have
\begin{equation} \label{eq32b}
\lim_{n \rightarrow \infty} \left| \underset{A \in \M_{n}(\Z_p)}{\Prob} \begin{pmatrix}
\cok(P_{ij}(A)) \cong H_{ij} \text{ for} \\ 
1 \leq i \leq l, \, 1 \leq j \leq k_i
\end{pmatrix} 
- \prod_{i=1}^{l} \left ( \underset{A_i \in \M_{n}(\Z_p)}{\Prob} \begin{pmatrix}
\cok(P_{ij}(A_i)) \cong H_{ij} \\ 
\text{ for } 1 \leq j \leq k_i
\end{pmatrix} \right ) \right| = 0.
\end{equation}
In particular, if we assume that the limit
$$
\lim_{n \rightarrow \infty} \underset{A_i \in \M_{n}(\Z_p)}{\Prob} \begin{pmatrix}
\cok(P_{ij}(A_i)) \cong H_{ij} \\ 
\text{ for } 1 \leq j \leq k_i
\end{pmatrix}
$$
converges for each $1 \leq i \leq l$, then we have
\begin{equation} \label{eq32c}
\lim_{n \rightarrow \infty}  \underset{A \in \M_{n}(\Z_p)}{\Prob} \begin{pmatrix}
\cok(P_{ij}(A)) \cong H_{ij} \text{ for} \\ 
1 \leq i \leq l, \, 1 \leq j \leq k_i
\end{pmatrix} 
= \prod_{i=1}^{l} \left ( \lim_{n \rightarrow \infty} \underset{A_i \in \M_{n}(\Z_p)}{\Prob} \begin{pmatrix}
\cok(P_{ij}(A_i)) \cong H_{ij} \\ 
\text{ for } 1 \leq j \leq k_i
\end{pmatrix} \right ).
\end{equation}
\end{conjecture}

%---------------------------------------------------------
%---------------------------------------------------------
\section{Independence phenomenon} \label{Sec4}

A special case of Theorem \ref{mainthm1} says that the distributions of $\cok (A)$ and $\cok (A+I_n)$ for a random $A \in \M_n(\Z_p)$ become independent as $n \rightarrow \infty$. It is natural to ask whether such independence phenomenon holds for the distributions of $\cok (A)$ and $\cok (A+B_n)$, where $\left \{ B_n \right \}_{n \geq 1}$ is a given sequence of matrices such that $B_n \in \M_n(\Z_p)$. 
In this section, we prove the following theorem that gives an answer to the question.

\begin{theorem} \label{mainthm2}
Let $\left \{ B_n \right \}_{n \geq 1}$ be any sequence of matrices such that $B_n \in \M_n(\Z_p)$ and
\begin{equation} \label{eq4b}
\lim_{n \rightarrow \infty} (n - \log_p n - r_p(\cok(B_n)) ) = \infty.
\end{equation}
Then for every finite abelian $p$-groups $H_1$ and $H_2$, we have
\begin{equation} \label{eq4a}
\lim_{n \rightarrow \infty} \underset{A \in \M_{n}(\Z_p)}{\Prob}
\begin{pmatrix}
\cok(A) \cong H_1 \text{ and } \\ 
\cok(A+B_n) \cong H_2
\end{pmatrix} = \prod_{j=1}^{2} \left ( \frac{1}{\left | \Aut_{\Z_p}(H_j) \right |} \prod_{i=1}^{\infty}(1-p^{-i}) \right ).
\end{equation}
\end{theorem}

For given $H_1$ and $H_2$, the following notations will be used throughout this section. 
%---------------------------------------------------------------
\begin{itemize}
    \item For $j \in \left\{ 1, 2 \right\}$, $\displaystyle c(H_j) :=  \frac{1}{\left | \Aut_{\Z_p}(H_j) \right |} \prod_{i=1}^{\infty}(1-p^{-i})$.
    
    \item For a given sequence $\left \{ B_n \right \}_{n \geq 1}$, denote
    $r_n  := r_p(\cok(B_n))$.
    
    \item For $n>2r>0$, let $\widetilde{\M}_{n, r}(\Z_p) \subset \M_n(\Z_p)$ be as in Section \ref{Sub22}. For $n>r>0$, denote
$$
\M^{1}_{n, r}(\Z_p) := \left \{ \begin{pmatrix}
A_1 & A_2 \\ 
A_3 & A_4
\end{pmatrix} \in \M_{r+(n-r)}(\Z_p) : (A_3)_{*1} = \begin{pmatrix}
1 & 0 & \cdots & 0
\end{pmatrix}^T \right \} \subset \M_n(\Z_p).
$$   

    \item For $n > r \geq 0$ and $X \in \M_{r}(\Z_p)$, denote
\begin{equation*}
\begin{split}
P_{n, r}(X) & := \underset{A \in \M_{n}(\Z_p)}{\Prob} (\cok (A) \cong H_1 \text{ and } \cok \left (A + \begin{pmatrix}
X & O \\ 
O & I_{n-r}
\end{pmatrix} \right ) \cong H_2), \\
P_{n, r}^{+} & := \underset{X \in \M_r(\Z_p)}{\sup} P_{n, r}(X), \\
P_{n, r}^{-} & := \underset{X \in \M_r(\Z_p)}{\inf} P_{n, r}(X).
\end{split}
\end{equation*}
    
    \item $\widetilde{P}_{n, r}(X)$, $P^{1}_{n, r}(X)$, $\widetilde{P}^{\pm}_{n,r}$ and $P^{1, \pm}_{n,r}$ are defined by the same way.
\end{itemize}

First we prove the following proposition, which is a weaker version of Theorem \ref{mainthm2}.

%---------------------------------------------------------------
\begin{proposition} \label{prop42}
The equation (\ref{eq4a}) holds for every $H_1$ and $H_2$ if
\begin{equation} \label{eq4c}
\lim_{n \rightarrow \infty} (n - 2 r_n ) = \infty.
\end{equation}
\end{proposition}

\begin{proof}
Fix $H_1$ and $H_2$. For $n > 2r > 0$ and $X \in \M_{r}(\Z_p)$, we obtain
\begin{equation} \label{eq4d}
\left| P_{n, r}(X) - \widetilde{P}_{n, r}(X) \right| \leq 1 - c_{n-r, r}
\end{equation}
by applying Lemma \ref{lem21b} as before. Now consider the probability $\widetilde{P}_{n, r}(X)$. Let
$$
A = \begin{pmatrix}
A_1 & A_2 & A_3 \\ 
I_r & A_4 & A_5 \\ 
O & A_6 & A_7 
\end{pmatrix} \in \widetilde{\M}_{n, r}(\Z_p)
$$
and 
$$
B = \begin{pmatrix}
X & O \\
O & I_{n-r}
\end{pmatrix} \in \M_{r+(n-r)}(\Z_p).
$$
By elementary operations, we can simultaneously transform $A$ and $B$ as follows:
\begin{equation*}
\small
\begin{split}
& \, (A, B) \\
\Rightarrow & \left ( \begin{pmatrix}
O & A_2-A_1A_4 & A_3-A_1A_5 \\ 
I_r & A_4 & A_5 \\ 
O & A_6 & A_7
\end{pmatrix}, \, \begin{pmatrix}
X & -A_1 & O \\ 
O & I_r & O \\ 
O & O & I_{n-2r}
\end{pmatrix} \right ) \\
\Rightarrow & \left ( \begin{pmatrix}
O & A_2-A_1A_4 & A_3-A_1A_5 \\ 
I_r & O & O \\ 
O & A_6 & A_7
\end{pmatrix}, \, \begin{pmatrix}
X & -A_1-XA_4 & -XA_5 \\ 
O & I_r & O \\ 
O & O & I_{n-2r}
\end{pmatrix} \right ) \\
\Rightarrow & \left ( \begin{pmatrix}
O & A_2' = A_2-A_1A_4+XA_5A_6 & A_3'= A_3-A_1A_5+XA_5A_7 \\ 
I_r & O & O \\ 
O & A_6 & A_7
\end{pmatrix}, \, \begin{pmatrix}
X & -A_1' = -A_1-XA_4 & O \\ 
O & I_r & O \\ 
O & O & I_{n-2r}
\end{pmatrix} \right ) \\
=: & \, (A', B').
\end{split}
\end{equation*}
Now we have
$$
\cok(A) \cong \cok (A') \cong \cok \begin{pmatrix}
A_2' & A_3' \\
A_6 & A_7 \\
\end{pmatrix}
$$
and
\begin{equation*}
\begin{split}
\cok(A+B) & \cong \cok(A'+B') \\
& \cong \cok \begin{pmatrix}
X & A_2'-A_1' & A_3' \\ 
I_r & I_r & O \\ 
O & A_6 & A_7+I_{n-2r}
\end{pmatrix} \\
& \cong \cok \begin{pmatrix}
O & A_2'-A_1'-X & A_3' \\ 
I_r & I_r & O \\ 
O & A_6 & A_7+I_{n-2r}
\end{pmatrix} \\
& \cong \cok \left ( \begin{pmatrix}
A_2' & A_3' \\
A_6 & A_7 \\
\end{pmatrix} + \begin{pmatrix}
-A_1'' = -A_1'-X & O \\
O & I_{n-2r} \\
\end{pmatrix}  \right ).
\end{split}    
\end{equation*}
For given $A_i$ ($4 \leq i \leq 7$) and $X$, the matrices $A_1''$, $A_2'$, $A_3'$ are random and independent if $A_i$ ($1 \leq i \leq 3$) are random and independent. Therefore $\widetilde{P}_{n, r}(X)$ is independent of the choice of $X$ so the equation (\ref{eq4d}) implies that
\begin{equation} \label{eq4e}
\left| P_{n, r}(X) - P_{n, r}(I_r) \right| \leq 2(1 - c_{n-r, r})
\end{equation}
for every $X \in \M_r(\Z_p)$. 

Now consider a sequence $\left \{ B_n \right \}_{n \geq 1}$ which satisfies the condition (\ref{eq4c}). The Smith normal form of $B_n$ is of the form $\displaystyle \begin{pmatrix}
X_{n} & O\\ 
O & I_{n-r_n}
\end{pmatrix}$ for some $X_n \in \M_{r_n}(\Z_p)$. 
The equation (\ref{eq4e}) implies that
$$
\left | P_{n, r_n}(X_n) - P_{n, r_n}(I_{r_n}) \right | \leq 2(1-c_{n-r_n, r_n})
$$
when $n>2r_n$. The assumption $\displaystyle \lim_{n \rightarrow \infty} (n- 2r_n) = \infty$ implies that
$$
\lim_{n \rightarrow \infty} 2(1-c_{n-r_n, r_n})=0.
$$
Therefore 
$$
\lim_{n \rightarrow \infty} \underset{A \in \M_{n}(\Z_p)}{\Prob}
\begin{pmatrix}
\cok(A) \cong H_1 \text{ and } \\ 
\cok(A+B_n) \cong H_2
\end{pmatrix}
= \lim_{n \rightarrow \infty} P_{n, r_n}(X_n)
= \lim_{n \rightarrow \infty} P_{n, r_n}(I_{r_n})
= c(H_1)c(H_2),
$$
where the last equality follows from Theorem \ref{mainthm1}.
\end{proof}

The Smith normal form of each $B_n$ tells us that Theorem \ref{mainthm2} is true if the equation
\begin{equation*}
\lim_{n \rightarrow \infty} P_{n, r_n}^+
= \lim_{n \rightarrow \infty} P_{n, r_n}^-
= c(H_1)c(H_2)
\end{equation*}
holds whenever the condition (\ref{eq4b}) is satisfied. The next proposition provides a relation between $P_{n, r}^{\pm}$ and $P_{n-1, r-1}^{\pm}$.

%---------------------------------------------------------------
\begin{proposition} \label{prop43}
For $n>r \geq 1$, we have
\begin{equation} \label{eq4f}
P_{n-1, r-1}^{-} - d_{n, r}
\leq P_{n,r}^{-} 
\leq P_{n, r}^{+} 
\leq P_{n-1, r-1}^{+} + d_{n, r}
\end{equation}
for $\displaystyle d_{n,r} := (1-c_{n-r, 1}) + (1 - c_{r,1} ) = \frac{1}{p^{n-r}} + \frac{1}{p^r}$.
\end{proposition}

\begin{proof}
Consider $A = \begin{pmatrix}
A_1 & A_2\\ 
A_3 & A_4
\end{pmatrix} \in \M_{r+(n-r)}(\Z_p)$. Applying Lemma \ref{lem21b} to the first column of $A_3$, we have
\begin{equation} \label{eq4g}
P_{n, r}^{+} \leq P_{n, r}^{1, +} + (1 - c_{n-r, 1}).
\end{equation}
Let 
$$
A = \begin{pmatrix}
A_1 & A_2 & A_3 & A_4\\ 
1 & A_5 & A_6 & A_7\\ 
O & A_8 & A_9 & A_{10}
\end{pmatrix} \in \M^{1}_{n, r}(\Z_p) \subset \M_{(r+1+(n-r-1)) \times (1+(r-1)+1+(n-r-1))}(\Z_p)
$$ 
and 
$$
X = \begin{pmatrix}
X_1 & X_2
\end{pmatrix} \in \M_{r \times (1+(r-1))}(\Z_p).
$$ 
By elementary operations, we can simultaneously transform $A$ and $\displaystyle \begin{pmatrix}
X & O \\ 
O & I_{n-r}
\end{pmatrix}$ as follows:
\begin{equation*}
\small
\begin{split}
& \, (A, \begin{pmatrix}
X & O \\ 
O & I_{n-r}
\end{pmatrix}) \\
\Rightarrow & \, \left ( \begin{pmatrix}
O & A_2' & A_3' & A_4'\\ 
1 & A_5 & A_6 & A_7\\ 
O & A_8 & A_9 & A_{10}
\end{pmatrix}, 
\begin{pmatrix}
X_1 & X_2 & -A_1 & O\\ 
O & O & 1 & O\\ 
O & O & O & I_{n-r-1}
\end{pmatrix}
 \right ) \\
\Rightarrow & \, \left ( \begin{pmatrix}
O & A_2' & A_3' & A_4'\\ 
1 & O & O & O\\ 
O & A_8 & A_9 & A_{10}
\end{pmatrix}, 
\begin{pmatrix}
X_1 & X_2-X_1A_5 & -A_1' & -X_1A_7\\ 
O & O & 1 & O\\ 
O & O & O & I_{n-r-1}
\end{pmatrix}
 \right ) \\
\Rightarrow & \, \left ( \begin{pmatrix}
O & A_2'' & A_3'' & A_4''\\ 
1 & O & O & O\\ 
O & A_8 & A_9 & A_{10}
\end{pmatrix}, 
\begin{pmatrix}
X_1 & X_2-X_1A_5 & -A_1' & O\\ 
O & O & 1 & O\\ 
O & O & O & I_{n-r-1}
\end{pmatrix}
 \right ).
\end{split}
\end{equation*}
%---------------------------------------------------------------
Now we have
$$
\cok (A) \cong \cok \left ( A' := \begin{pmatrix}
A_2'' & A_3'' & A_4'' \\ 
A_8 & A_9 & A_{10}
\end{pmatrix} \right ) 
$$
and

\begin{equation*}
\small
\begin{split}
\cok \left ( A + \begin{pmatrix}
X & O \\ 
O & I_{n-r}
\end{pmatrix} \right )  
& \cong \cok \begin{pmatrix}
X_1 & A_2''+X_2-X_1A_5 & A_3''-A_1' & A_4''\\ 
1 & O & 1 & O\\ 
O & A_8 & A_9 & A_{10}+I_{n-r-1}
\end{pmatrix} \\
& \cong \cok \begin{pmatrix}
O & A_2''+X_2-X_1A_5 & A_3''+A_1'' & A_4''\\ 
1 & O & 1 & O\\ 
O & A_8 & A_9 & A_{10}+I_{n-r-1}
\end{pmatrix} \\
& \cong \cok \begin{pmatrix}
A_2''+X_2-X_1A_5 & A_3''+A_1'' & A_4''\\ 
A_8 & A_9 & A_{10}+I_{n-r-1}
\end{pmatrix} \\
& = \cok \left ( A' + \begin{pmatrix}
X_2-X_1A_5 & A_1'' & O\\ 
O & O & I_{n-r-1}
\end{pmatrix} \right ).
\end{split}    
\end{equation*}
For given $X$ and $A_i$ ($5 \leq i \leq 10$), $A_j''$ ($1 \leq j \leq 4$) are random and independent if $A_j$ ($1 \leq j \leq 4$) are random and independent. Therefore $A'$, $A_1''$ and $A_5$ are random and independent if $A$ is random.

%---------------------------------------------------------------

Now define
\begin{equation*}
\begin{split}
P_{n-1, r}^{2}(X) & := \underset{\substack{A' \in \M_{n-1}(\Z_p) \\ A_1'' \in \M_{r \times 1}(\Z_p)}}{\Prob} (\cok (A') \cong H_1 \text{ and } \cok \left (A' + \begin{pmatrix}
X & A_1'' & O \\ 
O & O & I_{n-r-1}
\end{pmatrix} \right ) \cong H_2), \\
P_{n-1, r}^{3}(X) & := \underset{A' \in \M_{n-1}(\Z_p)}{\Prob} (\cok (A') \cong H_1 \text{ and } \cok \left (A' + \begin{pmatrix}
X & e_r & O \\ 
O & O & I_{n-r-1}
\end{pmatrix} \right ) \cong H_2)
\end{split}    
\end{equation*}
for $X \in \M_{r \times (r-1)}(\Z_p)$, $e_r := \begin{pmatrix}
0 & \cdots & 0 & 1
\end{pmatrix}^T \in \M_{r \times 1} (\Z_p)$ and 
$$
P_{n-1,r}^{i, +} := \underset{X \in \M_{r \times (r-1)}(\Z_p)}{\sup} P_{n-1, r}^{i}(X)
$$
for $i \in \left \{ 2, 3 \right \}$. Then we have the following.
\begin{itemize}
    \item The above reduction implies that $P_{n,r}^{1}(X) \leq P_{n-1,r}^{2, +}$ for every $X \in \M_{r \times (r-1)}(\Z_p)$ so
\begin{equation} \label{eq4h}
P_{n,r}^{1, +} \leq P_{n-1,r}^{2, +}.
\end{equation}

    \item For a random $Y \in \GL_r(\Z_p)$, $Ye_r$ is also random in $\M_{r \times 1}(\Z_p) \setminus \M_{r \times 1}(p \Z_p)$. Therefore
\begin{equation*}
\begin{split}
P_{n-1, r}^{2}(X) & \leq \underset{\substack{A' \in \M_{n-1}(\Z_p) \\ Y \in \GL_r(\Z_p)}}{\Prob} (\cok (A') \cong H_1 \text{ and } \cok \left (A' + \begin{pmatrix}
X & Ye_r & O \\ 
O & O & I_{n-r-1}
\end{pmatrix} \right ) \cong H_2) + (1-c_{r,1}) \\
& = \underset{\substack{A' \in \M_{n-1}(\Z_p) \\ Y \in \GL_r(\Z_p)}}{\Prob} (\cok (A') \cong H_1 \text{ and } \cok \left (A' + \begin{pmatrix}
Y^{-1}X & e_r & O \\ 
O & O & I_{n-r-1}
\end{pmatrix} \right ) \cong H_2) + (1-c_{r,1}) \\
& \leq P_{n-1,r}^{3, +} + (1 - c_{r, 1})
\end{split}    
\end{equation*}
for every $X \in \M_{r \times (r-1)}(\Z_p)$ so
\begin{equation} \label{eq4i}
P_{n-1,r}^{2, +} \leq P_{n-1,r}^{3, +} + (1 - c_{r, 1}).
\end{equation}
    
    \item For $X = \begin{pmatrix}
X_1 \\
X_2
\end{pmatrix} \in \M_{((r-1)+1) \times (r-1)}(\Z_p)$, we have
\begin{equation*}
\begin{split}
P_{n-1, r}^{3}(X) & = \underset{A' \in \M_{n-1}(\Z_p)}{\Prob} (\cok (A') \cong H_1 \text{ and } \cok \left (A' + \begin{pmatrix}
X_1 & O & O \\ 
X_2 & 1 & O \\ 
O & O & I_{n-r-1}
\end{pmatrix} \right ) \cong H_2) \\
& = \underset{A' \in \M_{n-1}(\Z_p)}{\Prob} (\cok (A') \cong H_1 \text{ and } \cok \left (A' + \begin{pmatrix}
X_1 & O & O \\ 
O & 1 & O \\ 
O & O & I_{n-r-1}
\end{pmatrix} \right ) \cong H_2) \\
& = P_{n-1, r-1}(X_1).
\end{split}    
\end{equation*}
This implies that
\begin{equation} \label{eq4j}
P_{n-1,r}^{3, +} = P_{n-1, r-1}^{+}.
\end{equation}
\end{itemize}
%---------------------------------------------------------------
By the equations (\ref{eq4g}), (\ref{eq4h}), (\ref{eq4i}) and (\ref{eq4j}), we have
\begin{equation*}
P_{n, r}^{+} \leq P_{n-1, r-1}^{+} + d_{n,r}.
\end{equation*}
The proof of the inequality $P_{n-1, r-1}^{-} - d_{n, r}
\leq P_{n,r}^{-}$ is exactly the same.
\end{proof}

%---------------------------------------------------------------
Now we prove the main result of this section.

\vspace{3mm}

\noindent \textit{Proof of Theorem \ref{mainthm2}.} Proposition \ref{prop42} implies that the formula (\ref{eq4a}) holds whenever 
$$
\lim_{n \rightarrow \infty} (n - 2r_n) = \infty.
$$
Therefore we may assume that $\displaystyle r_n \geq  \frac{2n}{5}$ for every $n \geq 1$. For $\displaystyle y_n := \left \lfloor \frac{3r_n-n}{2} \right \rfloor \geq \left \lfloor \frac{n}{10} \right \rfloor$, we have
\begin{equation} \label{eq4k}
\begin{split}
P_{n, r_n}^{+} 
& \leq P_{n-y_n, r_n-y_n}^{+} + \sum_{i=0}^{y_n-1} d_{n-i, r_n-i} \\
& \leq P_{n-y_n, r_n-y_n}^{+} + \frac{y_n}{p^{n-r_n}} + \frac{2}{p^{r_n-y_n+1}} \\
& \leq P_{n-y_n, r_n-y_n}^{+} + \frac{n}{p^{n-r_n}} + \frac{2}{p^{\frac{n-r_n}{2}}}
\end{split}    
\end{equation}
by Proposition \ref{prop43}. Now we have the following. 
\begin{itemize}
    \item The condition (\ref{eq4b}) implies that
\begin{equation} \label{eq4l}
\lim_{n \rightarrow \infty} \left ( \frac{n}{p^{n-r_n}} + \frac{2}{p^{\frac{n-r_n}{2}}} \right ) = 0.
\end{equation}

\item $\displaystyle (n-y_n)-2(r_n-y_n) \geq \frac{n-r_n-1}{2}$ so Proposition \ref{prop42} implies that
\begin{equation} \label{eq4m}
\lim_{n \rightarrow \infty} P_{n-y_n, r_n-y_n}^{+} = c(H_1)c(H_2).
\end{equation}
\end{itemize}
By the equations (\ref{eq4k}), (\ref{eq4l}) and (\ref{eq4m}), we have
$$
\limsup_{n \rightarrow \infty} P_{n, r_n}^{+} 
\leq \limsup_{n \rightarrow \infty} P_{n-y_n, r_n-y_n}^{+}
= c(H_1)c(H_2).
$$
The inequality $\displaystyle \liminf_{n \rightarrow \infty} P_{n, r_n}^{-} \geq c(H_1)c(H_2)$ can be deduced by the same argument. $\square$ \\

%---------------------------------------------------------
We expect that the condition (\ref{eq4b}) of Theorem \ref{mainthm2} can be strengthened to $\displaystyle \lim_{n \rightarrow \infty} (n - r_n) = \infty$. Although we did not prove this, we prove its converse: the formula (\ref{eq4a}) does not hold if $n-r_n$ does not go to infinity as $n \rightarrow \infty$. Denote $\displaystyle \alpha_{p, k} := \prod_{i=1}^{k} (1-p^{-i})$ and $\displaystyle \alpha_{p, \infty} := \prod_{i=1}^{\infty} (1-p^{-i})$.

%---------------------------------------------------------
\begin{lemma} \label{lem44}
Fix $k \geq 0$ and denote $\displaystyle I_{n, k}  := \begin{pmatrix}
I_k & O\\ 
O & O
\end{pmatrix} \in \M_{k+(n-k)}(\Z_p)$ for $n \geq k$. Then
\begin{equation*}
\lim_{n \rightarrow \infty} \underset{A \in \M_{n}(\Z_p)}{\Prob}
\begin{pmatrix}
\cok(A) = 0 \text{ and } \\ 
\cok(A+I_{n, k}) = 0
\end{pmatrix} 
= \alpha_{p, \infty} \alpha_{p, k}.
\end{equation*}
\end{lemma}

\begin{proof}
Define
\begin{equation*}
 P_{n, k}  
:= \underset{A \in \M_{n}(\Z_p)}{\Prob}
\begin{pmatrix}
\cok(A) = 0 \text{ and } \\ 
\cok(A+I_{n, k}) = 0
\end{pmatrix} 
= \underset{A \in \M_{n}(\F_p)}{\Prob}
\begin{pmatrix}
\cok(A) = 0 \text{ and } \\ 
\cok(A+\overline{I_{n, k}}) = 0
\end{pmatrix}
\end{equation*}
and
$$
P'_{n, k} 
:= \lim_{n \rightarrow \infty} \underset{\substack{A_1, A_2 \in \M_{n \times k}(\F_p) \\ C \in \M_{n \times (n-k)}(\F_p)}}{\Prob}
\begin{pmatrix}
\cok \left(\begin{array}{@{}c|c@{}}
    A_1 & C \\
  \end{array}\right) = 0 \text{ and } \\ 
\cok \left(\begin{array}{@{}c|c@{}}
    A_1+A_2 & C \\
  \end{array}\right) = 0
\end{pmatrix}.
$$
Then Lemma \ref{lem21b} (applied to $A_2$) implies that
$$
\left| P_{n, k} - P'_{n, k} \right| \leq 1-c_{n, k}.
$$
A matrix $\begin{pmatrix}
A_1 & C
\end{pmatrix} \in \M_{n \times (k+(n-k))}(\F_p)$ has a trivial cokernel if and only if $\rank (C) = n-k$ and the columns of $A_1$ give an $\F_p$-basis of $\F_p^n/C \F_p^{n-k} \cong \F_p^k$. 
For a random $A_1 \in \M_{n \times k}(\F_p)$, the images of its columns in $\F_p^n/C \F_p^{n-k}$ are also random. 
Therefore
\begin{equation*}
\begin{split}
P'_{n, k} & = \underset{C \in \M_{n \times (n-k)}(\F_p)}{\Prob} (\rank (C) = n-k) \times \left ( \frac{\left |\GL_k(\F_p)  \right |}{\left |\M_k(\F_p)  \right |} \right )^2 \\
& = \frac{\alpha_{p, n}}{\alpha_{p, k}} \alpha_{p, k}^2 \\
& = \alpha_{p, n} \alpha_{p, k}
\end{split}
\end{equation*}
so
\begin{equation*}
\lim_{n \rightarrow \infty} P_{n, k}
= \lim_{n \rightarrow \infty} P'_{n, k}
= \alpha_{p, \infty} \alpha_{p, k}. \qedhere
\end{equation*}
\end{proof}

%--------------------------------------------------
\begin{proposition} \label{prop45}
If the formula (\ref{eq4a}) holds for $H_1=H_2=0$, then $\displaystyle \lim_{n \rightarrow \infty} (n - r_n) = \infty$.
\end{proposition}

\begin{proof}
Assume that the condition $\displaystyle \lim_{n \rightarrow \infty} (n - r_n) = \infty$ does not hold. Choose an integer $k \geq 0$ and a subsequence $\left \{ B_{w_n} \right \}$ of $\left \{ B_{n} \right \}$ such that 
$$
w_n - r_p(\cok(B_{w_n})) = k
$$
for every $n \geq 1$. The Smith normal form of $\overline{B_{w_n}} \in \M_{w_n}(\F_p)$ is $\overline{I_{w_n, k}} \in \M_{k+(w_n-k)}(\F_p)$ so
\begin{equation*}
\begin{split}
\lim_{n \rightarrow \infty} \underset{A \in \M_{w_n}(\Z_p)}{\Prob}
\begin{pmatrix}
\cok(A) =0 \text{ and } \\ 
\cok(A+B_{w_n}) = 0
\end{pmatrix}  & = \lim_{n \rightarrow \infty} \underset{A \in \M_{w_n}(\F_p)}{\Prob}
\begin{pmatrix}
\cok(A) =0 \text{ and } \\ 
\cok(A+\overline{I_{w_n, k}}) = 0 
\end{pmatrix} \\
& = \alpha_{p, \infty} \alpha_{p, k} \\
& \neq \alpha_{p, \infty}^2
\end{split}    
\end{equation*}
by Lemma \ref{lem44}. 
\end{proof}

\begin{remark} \label{rmk46}
In the proof of the proposition above, $\displaystyle \lim_{k \rightarrow \infty} \alpha_{p, \infty} \alpha_{p, k} = \alpha_{p, \infty}^2$ so the events $\cok(A)=0$ and $\cok(A+B_n)=0$ for a random $A \in \M_n(\Z_p)$ become independent as $n \rightarrow \infty$ whenever $\displaystyle \lim_{n \rightarrow \infty} (n - r_n) = \infty$. Since we consider the zero cokernels, the distribution does not change after the reduction modulo $p$. However, this argument cannot be applied to the nonzero cokernels. 

For example, a version of Theorem \ref{mainthm1} for a random matrix over $\F_p$ was proved by Boreico \cite[Theorem 3.8.18]{Bor16} (and independently by Cheong and Huang \cite[Theorem 2.10]{CH21}), but it is not easy to lift the distribution over $\F_p$ to the distribution over $\Z_p$. In fact, the main ingredient of the proof of Proposition \ref{prop11} in \cite{CK22} is the counting result for such lifting.
\end{remark}

%---------------------------------------------------------
%---------------------------------------------------------
\section*{Acknowledgments}

The author is supported by a KIAS Individual Grant (SP079601) via the Center for Mathematical Challenges at Korea Institute for Advanced Study. We thank Gilyoung Cheong for his helpful comments. 

%---------------------------------------------------------
%---------------------------------------------------------

\end{document}